\documentclass[12pt,twoside, final]{amsart}
\usepackage{amsmath,amsthm,amscd,amsfonts,amssymb,enumerate}
\usepackage{graphicx}
\usepackage{color}
\usepackage[colorlinks]{hyperref}

\newtheorem{theorem}{Theorem}[section]

\newtheorem{lemma}[theorem]{Lemma}
\newtheorem{corollary}[theorem]{Corollary}
\theoremstyle{definition}

\theoremstyle{remark}
\newtheorem{remark}[theorem]{Remark}
\numberwithin{equation}{section}
\newcommand{\Ke}{\operatorname{Ker}}
\newcommand{\rk}{\operatorname{rk}}
\newcommand{\Sing}{\operatorname{Sing}}

\newcommand{\Pic}{\operatorname{Pic}}

\begin{document}

\title[Martens--Mumford's Theorems for ...]{
Martens--Mumford's Theorems for Brill-Noether Schemes arising from very ample Line bundles}

\author[A. Bajravani]{Ali Bajravani}
\address[]{Department of Mathematics, Faculty of Basic Sciences, Azarbaijan Shahid Madani University, Tabriz, I. R. Iran.\\
P. O. Box: 53751-71379}

\email{bajravani@azaruniv.edu}

\maketitle

\begin{abstract}

Tangent spaces of $V_{d}^{r}(L)$'s, specific subschemes of $C_{d}$ arising from various line bundles $L$ on $C$, are described. Then we proceed to prove Martens theorem for these schemes, by which we determine curves $C$, which for some very ample line bundle $L$ on $C$ and some integers $r$ and $d$ with $d\leq h^{0}(L)-2$, the scheme $V_{d}^{r}(L)$ might attain its maximum dimension. \\
   
\noindent \textbf{Keywords:}  Marthens--Mumford's Theorems; Symmetric Products; Very Ample Line Bundle.

\noindent \textbf{MSC(2010):}  14H99; 14H51.
\end{abstract}

\section{\bf Introduction}
 For a smooth projective algebraic curve $C$ of genus $g$ let $C_{d}$ denote $d$-th symmetric product of $C$. For non-negative integers $r$, $d$ with $0\leq 2r\leq d$, the closed subscheme $C^{r}_{d}\subset C_{d}$ parameterizes the locus of divisors of degree $d$ and with space of global sections  of dimension at least $r+1$, on $C$.  The scheme $C^{r}_{d}$ can be described locally as the locus of divisors for which the rank of the Brill-Noether matrix does not exceed $d-r$. This globalizes to give the well known scheme structure on $C^{r}_{d}$, as the $(d-r)$-th determinantal scheme of the morphism:
 $$u_{*}: \Theta_{C_{d}}\rightarrow \Theta_{Pic^{d}(C)}.$$
 See \cite[Chapter $4$]{ACGH}  for details and notations.
 
 For an arbitrary line bundle $L$ on $C$ this construction can be generalized to give similar subshemes of $C_{d}$, denoted by $V^{r}_{d}(L)$, and the usual Brill-Noether schemes, $C^{r}_{d}$'s, are special cases when $L$ is substituted by the canonical line bundle. See \cite[Chapter VII]{ACGH}. The schemes, $V^{r}_{d}(L)$'s, have been re-appeared recently in \cite{A-S}, where the authors study Koszul Cohomologies of Curves. As a byproduct of their interesting paper, they prove:
 \begin{lemma}\label{Lemma0}(Aprodu-Sernesi)
 If $L$ is a line bundle on $C$ and $d\geq 4$, then $V^{1}_{d}(L)$ is non-empty and of pure dimension $d-2$.
 \end{lemma}
 Lemma \ref{Lemma0}, although is not one of main themes of paper \cite{A-S}, basically is the starting point of our research. 
 We study the projective geometric aspects of $V^{r}_{d}(L)$'s.  
 First in section \ref{Tangent Space computations}, we describe the tangent spaces of $V^{r}_{d}(L)$. 
  As well $V^{r}_{d}(L)$ is described as intersection of $V^{r}_{d}(H)$'s, where $H$ moves on the set of sub-line bundles of $L$.

One of the essential ingredients in Brill-Noether Theory is the Martens, as well as Mumford's, Theorem.  See Theorems \ref{Martens} and \ref{Mumford}.  We prove Martens theorem for $V^{r}_{d}(L)$'s in section \ref{Martens--Mumford's theorems}, Theorem \ref{thm2}. This generalizes Theorem \ref{Martens} when $K_{C}$ is substituted by an arbitrary  very ample line bundle $L$. An existing proof of Martens theorem for $L=K_{C}$, as it can be found in \cite{ACGH}, uses tangent spaces of $C^{r}_{d}$'s together with the fact that
$C$ admits a finite number of theta charactersitics. The last part of this proof is not applicable when $L$ is different from the canonical line bundle. To overcome to this obstacle, when $C$ is non-hyper elliptic, we take a different approach which is based on taking an incidence correspondence and counting dimensions, see Theorem \ref{thm2}. 
As a byproduct we reobtain the dimension part of \cite[Lemma 2.1]{A-S} when $L$ is very ample.

During theorem \ref{thm3} we 
extend the well known Mumford's Theorem to $V^{r}_{d}(L)$'s. While we adopt a part of proof of the Mumford's Theorem somehow, our proof is essentially different from existing proof of Mumford's Theorem. Actually Lemmas \ref{lemma4} and \ref{lemma5}, as our basic tool, furnish the way to prove theorem \ref{thm3}.

   Remember that Keem's Theorem determines curves $C$, which for them under some circumstances $\dim C^{r}_{d}\geq d-r-2$, to be $4$-gonal. See for example \cite{ACGH} or \cite{Keem}. 
   Through remark \ref{Remark}, we see that occurance of two extra type of curves in Theorem \ref{thm3} goes back originally to Keem's Theorem.

 During, we follow notations of
  \cite{ACGH}. Particularly
 we denote by $V^{1}_{d}(L)$ what they denote by $V^{r-q+1}_{r-q+2}(L)$ or by $V^{r-1}_{r}(L)$ in \cite{A-S}.

\section{\bf Preliminaries and backgrounds}
For a smooth projective algebraic curve $C$, let $\pi_{1}$ and $\pi_{2}$ be the projections from $C\times C_{d}$ to $C$ and $C_{d}$ respectively.  
Then for a line bundle $L$ on $C$, the coherent sheaf 
$$\mathcal{L}:=(\pi_{2})_{*}(\mathcal{O}_{\Delta}\otimes \pi_{1}^{*}L),$$
where $\Delta \subset C\times C_{d}$ is the universal divisor of degree $d$, is a vector bundle of rank $d$ on $C_{d}$. Moreover for $D\in C_{d}$ we have the identifications 
\[\mathcal{L}_{D}\cong H^0(C, L\otimes \mathcal{O}_D) \cong H^0(C, L/L(-D)).\]
The natural map $\pi_{1}^*L\rightarrow \mathcal{O}_{\Delta}\otimes \pi_{1}^{*}L$ induced by restriction on $\Delta$, which is a map of vector bundles on $C\times C_{d}$, pushes forward via $\pi_{2}$ to a map of vector bundles on $C_{d}$:
\[f_{L} : H^0(C, L)\otimes \mathcal{O}_{C_{d}}\rightarrow \mathcal{L}.\]
Assuming $L$ to be a line bundle of degree $\delta$ and the space of global sections of dimension $s+1$,
 the map $f_{L}$ would be a map of vector bundles of ranks $s+1$ and $d$ respectively. For a non negative integer $r$ set
 \[V^{r}_{d}(L):=\lbrace D\in C_{d}\mid \rk(f_{L})_{D} \leq d-r \rbrace.\]
 The subscheme $V^{r}_{d}(L) \subset C_{d}$ parameterizes those effective divisors of degree $d$ 
 on $C$ that impose at most $d-r$ conditions on $\mid L \mid$, as well it is expected to be of dimension $d-r(s+1-d+r)$.\\
  For $L=K_{C}$, as it is commonly used in literature, the scheme $V^{r}_{d}(K_{C})$ will be denoted by $C^{r}_{d}$. Let
  $$
 \begin{array}{ccc}
 \alpha: C_{d}&\rightarrow & \Pic ^{d}(C)\\
 D & \mapsto &\mathcal{O}(D),
 \end{array} 
 $$
 be the Abel map and set $W^{r}_{d}(C):=\alpha (C^{r}_{d})$. Using the notion of Poincare line bundle, the subset $W^{r}_{d}(C)\subset \Pic ^{d}(C)$ admits the structure of a closed subscheme. Furthermore 
   Martens Theorem gives an upper bound for $\dim W^{r}_{d}(C)$, as well as Mumford's Theorem classifies curves $C$ for which $\dim W^{r}_{d}(C)$ attains its maximum value for some integers $r$ and $d$. 
   
\begin{theorem}\label{Martens}
(Martens) Let $C$ be a smooth curve of genus $g\geq 3$. Let $d$ be an integer such that $2\leq d\leq g-1$ and let $r$ be an integer such that $0<2r\leq d$. Then if $C$ is non hyper-elliptic, every component of $W^{r}_{d}(C)$ has dimension at most equal to $d-2r-1$.
 If $C$ is hyper-elliptic, then $\dim W^{r}_{d}(C)=d-2r$.
\end{theorem}
\begin{theorem} \label{Mumford}
(Mumford) Let $C$ be a smooth non-hyper elliptic curve of genus $g\geq 4$.
Suppose that there exist integers $r$ and $d$ such that $2\leq d\leq g-2$, $d\geq 2r>0$ and a component $X$ of $W^{r}_{d}(C)$ with
$\dim X=d-2r-1$. Then $C$ is either trigonal, bielliptic or a smooth plane quintic.
\end{theorem} 

\noindent  See \cite{ACGH} and \cite{Mumford} for proof of Theorems \ref{Martens} and \ref{Mumford}.
Based on Mumford's Theorem, we call a curve $C$ of Mumford's type if either it is bielliptic, trigonal or a smooth plane quintic. 
 
 \section{\bf Tangent Space computations for $V^{r}_{d}(L)$}\label{Tangent Space computations}
  
Let $\lbrace \omega^{L}_{1}, \omega^{L}_{2}, \cdots \omega^{L}_{s+1}  \rbrace$ be a basis for $H^{0}(C, L)$ and $\phi^{L}: C_{d}\rightarrow M_{d\times (s+1)}$ be the map defined by
\[\phi^{L}(\sum_{i=1}^{i=d} q_{i})=(\omega^{L}_{t}(q_{i}))_{i,t}\]
where $M_{d\times (s+1)}$
is the space of $d$ by $
(s+1)$- matrices. For $D\in C_{d}$ setting $A=\phi^{L}(D)$ the restriction map
$\alpha_{L}: H^0(C, L) \rightarrow H^0(C, L\otimes \mathcal{O}_{D})$,
is represented by $A$. As a consequence of this fact; for $\nu \in T_{D}(C_{d})$ one might identify $\phi^{L}_{*}(\nu).\ker(A)$
with $\beta_{L}(\nu\otimes H^0(C, L(-D)))$, where $\beta$ is the cup product homomorphism
\[H^0(C, \mathcal{O}_{D}(D))\otimes H^0(C, L(-D)) \rightarrow H^0(C, L\otimes\mathcal{O}(D)).\]

\begin{lemma}\label{lem2}
(a) If $D$ belongs to $V^r_{d}(L)\setminus V^{r+1}_{d}(L)$, the tangent space to $V^r_{d}(L)$ at $D$
is 
\[T_{D}(V^{r}_{d}(L))=(Im (\alpha_{L}\mu_{0}^{L}))^{\bot}\]
where 
  $\mu_{0}^{L}$ is the cup product map
  \[\mu_{0}^{L}: H^0(C, \mathcal{O}(D))\otimes H^0(C, L(-D)) \rightarrow H^0(C, L).\]
(b) If $D\in V^{r+1}_{d}(L)$ then 
$T_{D}(V^r_{d}(L))= H^0(C, L\otimes \mathcal{O}_{D})$.
Particularly, if $V^r_{d}(L)$ has the expected dimension and $d< s+1+r$, then $D\in \Sing(V^r_{d}(L))$. 
\end{lemma}
\begin{proof}
(a) 
This is a repetition of discussions in pages $161-162$ of \cite{ACGH}.\\
(b) For $D\in C_{d}$, we have $D\in V^{r+1}_{d}(L)$  if and only if $\phi^{L}(D)\in  M_{d\times (s+1)}(r+1)$. Now the equality  $T_{\phi^{L}(D)}M_{d\times (s+1)}=M_{d\times (s+1)}$, which leads to the assertion, is a well known fact.   
See \cite[Chapter II-Section 2]{ACGH}.
\end{proof}
\begin{theorem}\label{thm2.2}
The scheme $V_{d}^{r}(L)$ is smooth at $D\in V_{d}^{r}(L)\setminus V_{d}^{r+1}(L)$ and has the expected dimension 
$d-r\cdot (s+1-(d-r))$
if and only if $ \mu_{0}^{L}$ is injective.
\end{theorem}
\begin{proof}
 Since $\Ke(\alpha_{L})=H^{0}(C, L(-D))\subset Im(\mu_{0}^{L})$
one has 
$$\begin{array}{ccc}
&&\!\!\!\!\!\!\!\!\!\!\!\!\!\!\!\!\!\!\!\!\!\!\!\!\!\!\!\!\!\!\!\!\!\!\!\!\!\!\!\!\!\!\!\!\!\!\!\!\!\!\!\!\!\!\!\!\!\!\!\!\!\!\!\!\!\!\!\!\!\!\!\!\!\!\!\!\!\!\!\!\!\!\!\!\!\!\!\!\!\!\!\!\!\!\!\!\!\!\!\!\!\!\!\!\!\!\!\!\!\!\!\!\!\!\!\!\!\!\!\!\!\!\!\!\!\!\!\!\!\!\!\!\!\!\!\!\!\!\!\!\!\!\!\!\!\!\!\!\!\!\!\!\!\!\!\!\dim T_{D}[V^{r}_{d}(L)]\\
&=&\!\!\!\!\!\!\!\!\!\!\!\!\!\!\!\!\!\!\!\!\!\!\!\!\!\!\!\!\!\!\!\!\!\!\!\!\!\!\!\!\!\!\!\!\!\!\!\!\!\!\!\!\!\!\!\!\!\!\!\!\!\!\!\!\!\!\!\!\!\!\!\!\!\!\!\!\!\!\!\!d-(\dim Im(\mu_{0}^{L})-\dim \Ke(\alpha_{L}))\\
&=&\!\!\!\!\!\!\!\!\!\!\!\!\!\!\!\!\!\!\!\!\!\!\!\!\!\!\!\!\!\!\!\!\!\!\!\!\!\!\!\!\!\!\!\!\!\!\!\!\!\!\!\!\!\!\!\!\!\!\!\!\!\!\!\!\!\!\!\!\!\!\!\!\!\!\!\!d- r\cdot h^{0}(C, L(-D)) + \dim \Ke \mu_{0}^{L}\\
&=&\!\!\!d-r\cdot(s+1-(d-r))-r\cdot (d-r- \dim Im(\alpha_{L}))+ \dim \Ke \mu_{0}^{L}.
\end{array}$$
This implies the assertion.
\end{proof}
\begin{lemma}\label{lemma3}
For a line bundle $L$ and integers $r, d$ with $h^{0}(L)> d-r+1$, we have:
\[ V^{r}_{d}(L)=\bigcap_{H \subseteq L}  V^{r}_{d}(H), \]
where $H$ moves on the set of sub-line bundles of $L$.
\end{lemma}
\begin{proof}
 If $D$ is a point of $\in V_{d}^{r}(H)$, then 
\[\dim (Im (e_{H}^{D}))=rank (e_{H}^{D})\leq d-r.\]
If $H$ is a sub-line bundle of the line bundle $L$, 
then 
 a diagram chasing shows $Im (e_{H}^{D}) \subseteq Im (e_{L}^{D})$. This implies that  
 $V^{r}_{d}(L) \subseteq  V^{r}_{d}(H)$.
 
 If a divisor $D\in C_{d}$ belongs to 
 $$[\bigcap_{H \subseteq L}  V^{r}_{d}(H)] \setminus V^{r}_{d}(L),$$
then for each $p\in C$ we have
$h^{0}(L(-p)(-D))=h^{0}(L(-D)),$
  which means that; any point of $C$ is a base point for $L(-D)$. Therefore any global section of 
  $L(-D)$ vanishes everywhere on $C$,  so  $H^{0}(L(-D))=0$. This together with the fact that for $p\in C$ we have $D\in V^{r}_{d}(L(-p))$, implies $h^{0}(L)\leq d-r+1$, which is absurd by our hypothesis
 $h^{0}(L)> d-r+1$.  
  
\end{proof}
\section{Martens--Mumford's theorems for $V^{r}_{d}(L)$}\label{Martens--Mumford's theorems}

In this section we assume that $L$ is a very ample line bundle on $C$. Therefore 
$ \psi_{L}: C\rightarrow \mathbf{P}(H^{0}(C, L)):=\mathbf{P}_{L}$,
the map induced by $L$, is an embedding.
Repeating the proof of \cite[Lemma 2.2]{A-S} we obtain:
\begin{lemma}\label{lemma0}
For a very ample line bundle $L$ on $C$ and an integer $d$ with $d\geq 4$, if $V^{r}_{d}(L)\neq\emptyset$, then no irreducible component of $V^{r}_{d}(L)$
is contained in $V^{r+1}_{d}(L)$.
\end{lemma}
A direct consequence of Lemma \ref{lemma0} is that the locally closed subscheme 
$$S^{r}_{d}(L):=V^{r}_{d}(L)\setminus V^{r+1}_{d}(L)$$
is dense in any irreducible component of $V^{r}_{d}(L)$.
\begin{theorem}\label{thm2} (a) Let $C$ be a hyper-elliptic curve and $L$ a line bundle on $C$
with the space of global sections of dimension $s+1$. Assume moreover that $d\leq s+1$. Then 
$V^{r}_{d}(L)$ is empty or irreducible of dimension $d-r$ according to whether $d<2r$ or $2r\leq d$, 
respectively.\\
(b) If $C$ is non hyper-elliptic and $L$ a very ample line bundle on $C$ with $d\leq h^{0}(L)-1$, then every component of
$V^{r}_{d}(L)$, has dimension at most equal to $d-r-1$.
\end{theorem}
\begin{proof}
(a) Without loss of generality we reduce to the case that $L$ is base point free, so $L=sg^{1}_{d}$. If $D\in C_{d}$ then $D\in V^{r}_{d}(L)\setminus V^{r+1}_{d}(L)$ if and only if $h^0(L(-D))=s+r+1-d$. 
If $h^0(D)=\bar{r}+1$ then $D=\bar{r}g^{1}_{2}+q_{1}+\cdots +q_{d-2\bar{r}}$ for some $q_{1},\cdots ,q_{d-2\bar{r}}$ on $C$ and $\bar{L}=(s-\bar{r})g^{1}_{2}$. Consider that the divisor 
$\phi_{\bar{L}}(q_{1})+\cdots +\phi_{\bar{L}}(q_{d-2\bar{r}})$ on the rational normal curve $\phi_{\bar{L}}(C)\subset \mathbf{P}(H^{0}((s-\bar{r})g^{1}_{2}))$, spans a $\mathbf{P}^{d-2\bar{r}-1}$ in $\mathbf{P}(H^{0}((s-\bar{r})g^{1}_{2}))$. 
Therefore we have $h^0(L(-D))=s+\bar{r}+1-d$. An immediate consequence of this observation is that 
$V^{r}_{d}(L)=\emptyset$ provided $2r>d$ and  $V^{r}_{d}(L)\setminus V^{r+1}_{d}(L)=C^{r}_{d}\setminus C^{r+1}_{d}$ for $2r\leq d$. This concludes (a).

\noindent To prove (b), consider an incidence correspondence $\mathcal{H}$ as   
\[\mathcal{H}=\lbrace (H, D): D\subset H\cap C_{L} \rbrace \subset (\mathbf{P}_{L})^*\times C_{d},\]
 where $H$ is a hyperplane in 
 $\mathbf{P}_{L}$.
  Assuming $V^{r}_{d}(L)\neq \varnothing$, let $W$ be an
irreducible component of $V^{r}_{d}(L)$ of maximal dimension.
Using Lemma \ref{lemma0}, we have $W=\overline{(V_{d}^{r}(L)\setminus V_{d}^{r+1}(L))\cap W}$. 
Consider the subscheme $\Sigma$ of $\mathcal{H}$ as:
 \[\Sigma:=\pi_{2}^{-1}((V_{d}^{r}(L)\setminus V_{d}^{r+1}(L))\cap W) \subset \mathcal{H},\]
 where $\pi_{2}$ is the second projection from $ (\mathbf{P}_{L})^*\times C_{d}$ composed with the inclusion.
 
For $D\in C_{d}$ it is easy to see that either
$\vert D \vert \cap V_{d}^{r}(L)=\varnothing$ or
 $\vert D \vert \subseteq V_{d}^{r}(L)$.
This together with the geometric Riemann-Roth Theorem, imply that a
point of $V^{r}_{d}(L) \setminus V^{r+1}_{d}(L)$ lies in a $\mathbf{P}^{s}\subset \mathbf{P}_{L}$, where
$s=h^{0}(L)-1-h^{0}(L(-D))=d-r-1$.
Therefore the generic fiber of 
$\Sigma \rightarrow  (V_{d}^{r}(L)\setminus V_{d}^{r+1}(L))\cap W$
 is a $\mathbf{P}^{m}$, where $m=h^{0}(L)-d+r-1$. So we obtain
 \[\dim (\Sigma)=\dim( (V_{d}^{r}(L)\setminus V_{d}^{r+1}(L))\cap W)+h^{0}(L)-d+r-1.\]
Consider now that the projection on the second factor is a finite to one and non-surjecting map by the general position Lemma. This consideration implies that 
$\dim (\Sigma)\leq h^{0}(L)-2. $
Summarizing we get
\[\dim( V^{r}_{d}(L))=\dim( V^{r}_{d}(L) \setminus V^{r+1}_{d}(L))\leq d-r-1.\]

\end{proof}
Using Theorem \ref{thm2}, we recover the dimension part of \cite[Lemma 2.1]{A-S} for very ample line bundles.
\begin{corollary}\label{Cor}
Assume that $L$ is a very ample line bundle on $C$ with $h^{0}(L)=d+1\geq 4$. Then $ V^{1}_{d}(L)$, if non empty, is of dimension $d-2$.
\end{corollary}
To obtain Mumford's Theorem on $V^{r}_{d}(L)$'s, Theorem \ref{thm3}, we need the next couple of lemmas.
\begin{lemma}\label{lemma4}
Assume that $L$ is a very ample line bundle on $C$ such that for some integers $r\geq 2$ and $d$ with $d\leq h^0(L)-2$ we have $\dim V^{r}_{d}(L)=d-r-1$, then $V^{r-1}_{d-1}(L)$ is of dimension $d-r-1$ too.
\end{lemma}
\begin{proof}
We assume that $V^{r}_{d}(L)$ is irreducible, since otherwise we can substitute it with an irreducible component. As we mentioned in proof of Theorem \ref{thm2}, for $D\in C_{d}$ we have either 
$\mid D \mid \cap V^{r}_{d}(L)=\emptyset$ or $\mid D \mid \subseteq V^{r}_{d}(L)$. For general $q\in C$ from the equality
$$V^{r}_{d}(L)=\bigcup_{p\in C}(p+C_{d-1})\cap V^{r}_{d}(L), $$
we obtain $\dim (q+C_{d-1})\cap V^{r}_{d}(L)=\dim V^{r}_{d}(L)-1$. In fact for general $q\in C$ the equality $V^{r}_{d}(L)=(q+C_{d-1})\cap V^{r}_{d}(L)$ implies that $q$, being a general point of $C$, is a base point of 
$\mid D \mid$ for each $D\in V^{r}_{d}(L)$. In other words each global section $\sigma$ of $\mathcal{O}(D)$ vanishes at $q$. 
Therefore  $\sigma$ vanishes everywhere on $C$. So $H^{0}(D)=0$ which by effectivity of $D$ is absurd, proving the equality $\dim (q+C_{d-1})\cap V^{r}_{d}(L)=\dim V^{r}_{d}(L)-1$. 

To establish the Lemma, consider that removing $q$ from the series in $V^{r}_{d}(L)$ we obtain $(1+(d-r-1)-1)$-dimensional family of divisors $\bar{D}$ belonging to $ V^{r-1}_{d-1}(L)$. This together with Theorem \ref{thm2}, imply the assertion. 
\end{proof}
\begin{lemma}\label{lemma5}
Let $L$ be a base point free line bundle on $ C$ such that for general $p\in C$, the line bundle $L(-p)$ is very ample.
 Assume moreover that $V$ is an irreducible component  of $V^{1}_{d}(L)$ and $d=h^{0}(L)-2\geq 4$.
Then 
for general $p\in C$, no irreducible component $W_{p}$ of $V^{1}_{d}(L(-p))$ 
 coincides on $V$. 
\end{lemma}
\begin{proof}
We prove that a general divisor $D$ in $V$ fails to be a general member of a component of $V^{1}_{d}(L(-p))$. 

For each $D\in V$,
 there exists $p\in C$ such that the subset $(p+C_{d-1})\cap V^{1}_{d}(L)$ is non-empty. For general $D$ in $V$, since $\dim V\geq 2$, this $p$ has to move in an open subset of $C$.  
 Otherwise for general $D\in V$ removing $p$ from the divisors of $V$ we obtain $\dim V^{1}_{d-1}(L(-p))=d-2$, contradicting Lemma (2.1) of \cite{A-S}.

If a general divisor $D$ in $V$ turns to be a general member of a component $W_{p}$ of $V^{1}_{d}(L(-p))$; then for general $ q\in C$ the divisors of type $D-p+q$, belonging to $V^{1}_{d}(L(-p))$, lie on $W_{p}$. This 
by genericity of $D$ and $q$ is impossible,
which completes the proof of Lemma.

\end{proof}
\begin{theorem}\label{thm3}
Assume that $C$ is a smooth projective non-hyper elliptic curve of genus $g$ with $g\geq 9$.
If for some very ample line bundle $L$ on $C$ there exist integers $r$, $d$ with $d\leq h^0(L)-2$ and a component $X$ of $V^{r}_{d}(L)$ with 
$\dim X=d-r-1$, then $C$ is one of the following type:
\begin{itemize}
\item[$\bullet$] A  bi-elliptic curve,
\item[$\bullet$] A 3-gonal curve,
\item[$\bullet$] A 4-gonal curve, or
\item[$\bullet$] A space curve of degree 7.
\end{itemize}
\end{theorem}   
\begin{proof}
Everywhere in this proof to simplify the notations, we shall write $V^{r}_{d}(L)$ instead of $X$. 
Consider first that using Lemma \ref{lemma4} we can reduce to the case $r=1$ and we have $\dim V^{1}_{d}(L)=d-2$.
Let $d$ be the minimum integer for which for some line bundle $K$ one has $\dim V^{1}_{d}(K)=d-2$. 
Assume moreover that $L$ is a very ample line bundle with minimum $h^{0}(L)$ among those very ample line bundles $H$, which for them $V^{1}_{d}(H)$ is of dimension $d-2$. Under these circumstances, for a general $p\in C$ either $L(-p)$ fails to be very ample or 
 $d\geq h^{0}(L(-p))-1$.  The latter case implies $h^{0}(L)=d+2$.

 In the first case; for a general $p\in C$ there are $q_{1}, q_{2}\in C$ such that $h^{0}(L(-p)(-q_{1}-q_{2}))\geq h^{0}(L(-p))-1=h^{0}(L)-2$. Therefore $h^{0}(L)-h^{0}(L(-p-q_{1}-q_{2}))\leq 2=3-1$, so $p+q_{1}+q_{2}\in V^{1}_{3}(L)$. This, according to Theorem \ref{thm2}, asserts $\dim V^{1}_{3}(L)=1$. Therefore $d\leq 3$.

 Using very ampelness of $L$ we have $V^{1}_{2}(L)=\emptyset$, excluding the case $d=2$. 
 
We proceed analyzing the case $d=3$.
 If for a $D\in V^{1}_{3}$ one has  $h^0(D)=2$ then $C$ is trigonal, while if for each $D\in V^{1}_{3}$ we had
$h^0(D)=1$ then for $D_{1}$ and $D_{2}$ in $ V^{1}_{3}$, we'll have $h^0(D_{1}+D_{2})\in \lbrace 2, 3 \rbrace$.
Therefore two subcases occur.

 If $ h^0(D_{1}+D_{2})=2$ then $h^0(K(-D_{1}-D_{2}))=g-5$. 
For general points $p_{1}, p_{2},\cdots ,p_{g-8}$ on $C$ we have $h^0(K(-2D_{1}-2D_{2})(-\Sigma p_{i}))\geq 1$, therefore $\mid D_{1}+D_{2} \mid$
is contained in $\mid K(-D_{1}-D_{2})(-\Sigma p_{i}))\mid=\mid M \mid$. Let 
$\phi: C\rightarrow \mathbf{P}^3$
be the morphism defined by $M$ and consider that the morphism 
$\phi_{\mid D_{1}+D_{2}\mid}:C\rightarrow \mathbf{P}^1$,
 defined by $\mid D_{1}+D_{2} \mid$, is obtained by composing $\phi$ with a projection from a line $L$ in $\mathbf{P}^3$ which we can assume it pass through a smooth point of $\phi(C)$. Therefore 
\[(\deg \phi)(\deg \phi (C)-1)=6.\]
Up to the last equality either $\deg \phi=1$ in which case $\phi(C)$, as well as $C$, is a space septic or; $\deg \phi=2$  in which case $\phi(C)$ is a space quartic which has to be a normal elliptic curve. Therefore $C$ is a  double covering of a
normal elliptic space curve, which is the same as to be bi-elliptic. Lastly 
 $\deg \phi=3$ and $\phi(C)$ is a space cubic curve which has to be a rational normal curve. Therefore $C$ is a  triple covering of a rational 
normal space curve, which means that $C$ is trigonal.

 If we are in the second subcase; i.e. for general $D_{1}, D_{2}\in V^{1}_{3}$ we have $ h^0(D_{1}+D_{2})=3$, then 
$\dim C^{2}_{6} \geq 2$, so 
$$\dim W^{2}_{6} \geq 0=6-2\times 2 -2. $$
Applying Keem's Theorem, see \cite[page 200]{ACGH},  and its extension by Coppens to the cases $g=9, 10$
in \cite{C}, imply that 
$C$ is a $4$-gonal curve. 

Assume finally that $d=h^{0}(L)-2$. If for general $p\in C$ the line bundle $L(-p)$ fails to be very ample or if $d=3$, then our position is the same as previous case.\\
 We exclude the other case. Assume that for general $p\in C$ the line bundle $L(-p)$ is very ample with $d\geq 4$ and consider that if $V$ is an irreducible component of $V^{1}_{d}(L)$ with maximal dimension ($d-2$), then for each $p\in C$ using Lemma (2.1) from \cite{A-S}, an irreducible component of
$V^{1}_{d}(L(-p))$ coincides on $V$.
 This, according to Lemma \ref{lemma5}, is absurd.

\end{proof}
\begin{remark}\label{Remark}
(a) Very ampleness in Theorem \ref{thm2} is necessary. To see this; for $p\in C$ set $L=K(p)$ and consider that for general points $q_{1}, q_{2}$ on $C$ and $D\in C^{1}_{g-1}$, effective divisors linearly equivalent to divisors of type 
$D+q_{1}+q_{2}-p$
belong to $V^{1}_{g}(L)$. So $\dim V^{1}_{g}(L)\geq g-1 > g-1-1$.

(b)
For curves of Mumford's type, the bound in Theorem \ref{thm3} is achieved when $L=K_{C}$, as it is well known in the literature. For very ample canonical sub-line bundles $L\subseteq K_{C}$, 
 since by Lemma \ref{lemma3} $V^{r}_{d}(L)$ contains $C^{r}_{d}$, then $\dim V^{r}_{d}(L)$
attains the bound of Theorem \ref{thm3}. 

(c)  If $C$ is bielliptic and $L$ is a very ample sub-line bundle of $K_{C}$, then Lemma \ref{lem2} together with Theorem \ref{thm2}, imply maximum dimensionality of $V^{2}_{d}(L)$, i.e. $ \dim (V^{2}_{d}(L))=d-3$.\\ 
If $L$ is a very ample line bundle which is not a sub-line bundle of the canonical bundle,
  denote by
   $\epsilon : C\rightarrow E$ 
  the elliptic involution and assume that $L$ enjoys from the property that for some $p, q \in E$ one has 
\[h^0(L)-h^0(L(-p_{1}-p_{2}-q_{1}-q_{2}))\leq 3\]
where $\epsilon^{-1}(p)=\lbrace p_{1}, p_{2} \rbrace$ and  $\epsilon^{-1}(q)=\lbrace q_{1}, q_{2} \rbrace$.
Then for $R_{1}, \cdots , R_{t}\in C$ we have 
$p_{1}+p_{2}+q_{1}+q_{2}+R_{1}+\cdots +R_{t}\in V^{1}_{t+4}(L)$ and $\dim(V^{1}_{t+4}(L))=t+2$.
 Additionally for some general $\gamma \in E$, with an extra prescribed assumption 
\[h^0(L)-h^0(L(-p_{1}-p_{2}-q_{1}-q_{2}-\gamma_{1}-\gamma_{2}))\leq 4\]
for which $\epsilon^{-1}(\gamma)=\lbrace \gamma_{1}, \gamma_{2} \rbrace$, the divisors of type
$$p_{1}+p_{2}+q_{1}+q_{2}+\gamma_{1}+\gamma_{2}+R_{1}\cdots +R_{t}$$
 belong to $V^{2}_{t+6}(L)$. This again imply maximum dimensionality of $V^{2}_{t+6}(L)$, i.e. $ \dim (V^{2}_{t+6}(L))=t+3$.

(d) Assume that $C$ is a $4$-gonal curve with $p\in C$ a base point of $K(-g^{1}_{4})$. Consider the very ample line bundle $L=K(-p)$ and observe that; divisors of type $g^{1}_{4}+q_{1}+\cdots +q_{t}$ belong to 
$V^{2}_{t+4}(L)$. Therefore $t+1 \leq \dim V^{2}_{t+4}(L)$. This using Theorem \ref{thm2} implies $\dim V^{2}_{t+4}(L)=t+1$.

(e) If $C$ is a space curve of degree $7$, then projecting from an smooth point of $C$ to $\mathbf{P}^{2}$ we obtain a plane sextic, which is singular. Therefore $C$ is birational to a $4$-gonal curve. This, by example (d) implies that $\dim V^{1}_{d}(L)$ might attain its maximum value.

\end{remark}


\begin{thebibliography}{20}


\bibitem{A-S}
M. Aprodu, E. Sernesi;
\newblock Secant spaces and syzygies of special line bundles on curves,
 \newblock To appear in Algebra and Number Theory, arXiv:1403.2516.

\bibitem{ACGH}
E. Arbarello, M. Cornalba, Ph. Griffiths, J. Harris;
\newblock Geometry of Algebraic Curves,
I. Grundlehren 267(1985), Springer. 

\bibitem{C-S}
C. Ciliberto, E. Sernesi;
\newblock Singularities of the Theta Divisor and Families of Secant Spaces to a Canonical Curve,
 \newblock {\em  J. Algebra} \textbf{171}(1995), 867--893.

\bibitem{C}
M. Coppens;
\newblock Some remarks on the Scheme $W^{r}_{d}$,
 \newblock {\em   Annali di Matem. P. e A.} \textbf{97}(1990), 183--197. 

\bibitem{Keem}
C. Keem;
\newblock  A remark on the Variety of Special Linear Systems on an Algebraic Curve,
 Ph. D. Thesis, Brown University (1983).

\bibitem{Ma}
I. G. Macdonald;
\newblock  Symmetric Products of an Algebraic Curve,
 \newblock {\em Topology} \textbf{1}(1962), 319--343.

\bibitem{Martens}
H. H. Martens, On the Varieties of special Divisors on a Curve, J. Reine Angew. Math. 
\textbf{227}(1967), 111-120.

\bibitem{Mumford}
D. Mumford, Prym Varieties I, in Contributions to Analysis (L. V. Ahlfors, I. Kra, B. Maskit, L. Nireremberg, eds.), Academic Press, New York, 1974, pp. 325-350.




\end{thebibliography}
\end{document}